\documentclass[]{interact}

\usepackage{epstopdf}
\usepackage[caption=false]{subfig}

\usepackage[numbers,sort&compress]{natbib}
\bibpunct[, ]{[}{]}{,}{n}{,}{,}
\makeatletter
\def\NAT@def@citea{\def\@citea{\NAT@separator}}
\makeatother

\theoremstyle{plain}
\newtheorem{theorem}{Theorem}[section]
\newtheorem{lemma}[theorem]{Lemma}

\newtheorem{proposition}[theorem]{Proposition}

\theoremstyle{definition}

\newtheorem{example}[theorem]{Example}

\theoremstyle{remark}
\newtheorem{remark}{Remark}

\begin{document}


\title{Minimal convex majorants of functions and \\ Demyanov--Rubinov super(sub)differentials}

\author{
\name{Valentin~V. Gorokhovik\thanks{Email: gorokh@im.bas-net.by}}\thanks{ID: https://orcid.org/0000-0003-2447-5943}
\affil{Institute of Mathematics,
The National Academy of Sciences of Belarus, \\ Minsk, Belarus}
}

\maketitle

\begin{abstract}
The primary goal of the paper is to establish characteristic properties of (extended) real-valued functions defined on normed vector spaces that admit the representation as the lower envelope of their minimal (with respect of the pointwise ordering) convex majorants. The results presented in the paper generalize and extend the well-known Demyanov-Rubinov characterization of upper semicontinuous positively homogeneous functions as the lower envelope of exhaustive families of continuous sublinear functions to more larger classes of (not necessarily positively homogeneous) functions defined on arbitrary normed spaces. As applications of the above results, we introduce, for nonsmooth functions, a new notion of the Demyanov—Rubinov subdifferential at a given point, and show that it generalizes a number of known notions of subdifferentiability, in particular, the Fenchel-Moreau subdifferential of convex functions and the Dini-Hadamard (directional) subdifferential of directionally differentiable functions. Some applications of Demyanov-Rubinov subdifferentials to extremal problems are considered.
\end{abstract}


\begin{amscode}
49J52, 49K27, 26B40
\end{amscode}

\section{Introduction}

The known classical result (see, for instance, \cite{Bourb}) states that a function defined on a metric space is lower (upper) semicontinuous if and only if it can be represented as the upper (lower) envelope of a family of continuous functions. On the other hand, it is also well-known \cite{ET76}, that each lower semicontinuous convex function defined on a normed vector space is the upper envelope of a family of continuous affine function. The latter statement plays a crucial role in establishing duality results in convex analysis and optimization. Besides, it shows that particular classes of semicontinuous functions can be represented as the upper or lower envelope of families of elementary (in some sense) continuous functions. Studies in this direction have led to the development of various abstract theories of convexity \cite{KutRub72,KutRub,GRC,Rub99,Rub00,PalRol,Singer}.


In 1982 Demyanov and Rubinov \cite[Theorem 2.1]{DR82} (see also \cite[Lemma 4.3]{DR90} and \cite[Lemma 5.2]{DR95}) proved that a real-valued positively homogeneous (p.h.) function defined on a Hilbert space $X$ is lower semicontinuous on $X$ if and only if it can be represented as the upper envelope of a family of real-valued continuous superlinear functions. Symmetrically, a real-valued p.h. function defined on a Hilbert space $X$ is upper semicontinuous on $X$ if and only if it can be represented as the lower envelope of a family of real-valued continuous sublinear functions. In 2000 Uderzo \cite{Ud2000} extended the above characterizations of semicontinuous p.h. functions to those defined on uniformly convex Banach spaces. At last, in 2017 Gorokhovik \cite{Gor2017a} proved that each upper semicontinuous p.h. function defined on an arbitrary normed vector space can be represented as the lower envelope of a family of continuous sublinear functions.
A family of continuous sublinear (superlinear) functions whose lower (upper) envelope is equal to a given p.h. function $p:X \to {\mathbb R}$ was called \cite{DR82,DR90,DR95} \textit{an exhaustive family of upper convex (lower concave) approximations} of $p.$ The characteristic properties of exhaustive families of upper convex (lower concave) approximations of various classes of positively homogeneous functions defined on finite-dimensional spaces are recorded in \cite{GS2011}; see also \cite{GorTr2016}.

Using exhaustive families of upper convex and lower concave approximations and the classical Minkowski dualty, Demyanov \cite{Dem99,DR2000} introduced dual objects for upper  and lower semicontinuous p.h. functions, called upper and lower exhausters, respectively. In fact, the correspondence between p.h. functions and their exhausters extends the classical Minkowski duality to the class of lower semicontinuous p.h. functions that is essentially more larger then those of sublinear or even difference-sublinear functions. The drawback of this extension is a multiplicity of exhaustive families of upper convex (lower concave) approximations and upper (lower) exhausters corresponding to the same function.

The main results of the present paper concern extensions of the Demyanov-Rubinov characterization of upper (lower) semicontinuous p.h. functions as the lower envelopes of continuous sublinear (superlinear) functions to (not necessarily p.h.) functions defined on an arbitrary normed vector spaces. To realize these extensions we enlarge the class of `elementary' functions used as upper approximations from continuous sublinear ones to convex ones.

We mainly discuss lower envelope presentations of functions. Passing from$f$ to $-f,$ all results can be symmetrically reformulated for upper envelope presentations. However, there are issues in variational analysis when we have to use both the lower envelope presentation and the upper envelope one simultaneously. For this reason some important results after their proof for lower envelope presentations are reformulated (without a proof) for upper envelope ones.    

The paper is organized as follows.

In Section \ref{sec1} we present some preliminaries results concerning sets in real vector spaces. For such sets we introduce the concept of a convex component, by which we mean a maximal (with respect to inclusion) convex subset of a given set. The family of convex components of a set is a covering of this set and, in this sense, it relates with an exhaustive family of upper convex approximations of a function. We also prove that the recession cone of a set agrees with the intersection of the recession cones of all its convex components.
The counterpart of the concept of a convex component is the one of a convex complement of a set.

In Section \ref{sec2} we associate with each function $f:X \to \overline{\mathbb R}$ ($\overline{\mathbb R}:= {\mathbb R}\bigcup\{\pm\infty\}$ is the extended real line) defined on a real vector space $X$ the family $\Sigma^+(f)$ consisting of all minimal (in the sense of the pointwise ordering) convex majorants of $f.$  The notion of a minimal convex majorant of a function is closely related to the notion of a convex component of a set: a function $g: X \to \overline{\mathbb R}$ is a minimal convex majorant of  a function $f:X \to \overline{\mathbb R}$  if and only if the epigraph of $g$ is a convex component of the epigraph of $f.$ The family $\Sigma^+(f)$ is nonempty for any $l$-proper function $f:X \to \overline{\mathbb R}$ ($f$ is $l$-proper if $f \not \equiv +\infty$ and $f(x) > -\infty$ for all $x \in X$) and, moreover, in this case $f$ is the pointwise minimum of $\Sigma^+(f).$

In general, including the case, when $f$ is a real-valued function, among minimal convex majorants of $f$ can be such that take the value $+\infty$. However, as it is proved in Section \ref{sec3}, if a real-valued function $f:X \to {\mathbb R}$ defined on a normed vector space is Lipschitz continuous on the whole space $X,$ each minimal convex majorant of $f$ is Lipschitz continuous on $X$ as well and, consequently, does not take the value $+\infty$.

For each real-valued function $f:X \to {\mathbb R}$ defined on a real normed space $X$ we denote by the symbol $\Sigma^+_{Lip}(f)$ the subfamily of $\Sigma^+(f)$ consisting of all minimal convex majorants of $f$ that are Lipschitz continuous on $X.$ For Lipschitzian functions (Theorem \ref{th4.1}) $\Sigma^+_{Lip}(f)$ agrees with $\Sigma^+(f).$  In general case  (see Theorem~\ref{th4.4}) the subfamily $\Sigma^+_{Lip}(f)$ is nonempty if and only if $f$ is Lipschitz bounded from above, that is, if and only if $f$ is majorized by a function that is Lipschitz continuous on $X.$ The most significant property of $\Sigma^+_{Lip}(f)$ is that the lower envelope of $\Sigma^+_{Lip}(f)$ is the upper semicontinuous closure of the function $f.$ From the above observations we conclude (Theorem~\ref{th4.7}) that a function $f$ is upper semicontinuous and Lipschitz bounded from above
on $X$ if and only if the family $\Sigma^+_{Lip}(f)$ is nonempty and $f$ is the lower envelope of $\Sigma^+_{Lip}(f),$ that is,
$
f(x) = \inf\limits_{g \in \Sigma^+_{Lip}(f)} g(x)\,\,\text{for all}\,\,x \in X.
$

For p.h. functions the requirement of Lipschitz boundedness from above in the last statement can be omitted, since each upper semicontinuous p.h. function is bounded from above by the Lipschitz continuous function $x \to k\|x\|$ with a suitable number $k > 0.$
Observe also, that each minimal convex majorant of a p.h. function is sublinear. Thus, since each continuous sublinear function $\varphi:X \to {\mathbb R}$ is Lipschitz continuous, the family $\Sigma^+_{Lip}(p)$ corresponding to a p.h. function $p: X \to {\mathbb R}$ consists exclusively of continuous sublinear functions.
To emphasize these peculiarities  we denote the family of all minimal continuous sublinear majorants of a p.h. function $p: X \to {\mathbb R}$ by the symbol $S^+_C(p)$ instead of $\Sigma^+_{Lip}(p).$

Taking into account the above remarks, we get from Theorem~\ref{th4.7} the following characterization of upper semicontinuous p.h. functions: a p.h. function $p$ defined on a normed vector space $X$ is upper semicontinuous on $X$ if and only if the family
$S^+_{C}(p)$ of minimal continuous sublinear majorants of $p$ is nonempty and $p$ is the lower envelope of $S^+_{C}(p)$, or, in the Demyanov-Rubinov terminology, $S^+_{C}(p)$ is an exhaustive family of upper convex approximations of $p.$ This statement extends the Demyanov-Rubinov characterization of upper semicontinuous p.h. functions to arbitrary normed space settings. 

In the concluding section \ref{sec4} of the paper we introduce, applying the above characterizations of p.h functions to directional derivatives, a new notion of subdifferentiability (superdifferentiability) of an extended-real-valued function at a given point called the Demyanov-Rubinov subdifferential (superdifferential). For convex functions the Demyanov-Rubinov subdifferential coincides with the classical Fenchel--Moreau subdifferential in the sense of convex analysis \cite{ET76,Rock,HULem1,HULem2}. For nonconvex functions the Demyanov-Rubinov subdifferential contains the Dini--Hadamard (directional) subdifferential \cite{Penot1,Penot2,Ioffe17} as a (possibly empty) subset. Observe also, that a function is G\^{a}teaux differentiable \cite{BonSha} at some point if and only if both its Demyanov-Rubinov subdifferential and its Demyanov-Rubinov superdifferential at this point are the same one-element family consisting only of the G\^{a}teaux derivative. Moreover, both the Demyanov-Rubinov subdifferential and the Demyanov-Rubinov superdifferential contain continuous linear functions if and only if a function is G\^{a}teaux differentiable.

Some applications of Demyanov-Rubinov subdifferentials to extremal problems are considered.

The results presented here were partially announced in \cite{Gor2017a,Gor2017b}.

\section{Convex components and convex complements of a set}\label{sec1}

Let $X$ be a real vector space.

By \textit{a convex component} of a nonempty set $Q$ in $X$ we call a maximal (in the sense of inclusion) nonempty convex subset of $Q.$

The existence of convex components for an arbitrary nonempty set $Q$ follows from Zorn's lemma \cite{Kelley}. Indeed, since any one-point subset of $X$ is convex, the collection of convex subsets of $Q$ is nonempty. Besides, for any chain of convex subsets of $Q$ ordered by inclusion the union of its subsets also is a convex subset of $Q.$ Hence, the collection of convex subsets of $Q$ is inductively ordered by inclusion. Due to Zorn's lemma the family of maximal convex subsets (convex components) of the set $Q$ is nonempty and, moreover, for each convex subset of $Q$ there exists a maximal convex subset (a convex component) which contains it.

The family of all convex components of a set $Q$ will be denoted by $\sigma^+(Q).$

The next theorem summarizes the above observations.

\begin{theorem}\label{th1.1}
The family of convex components $\sigma^+(Q)$ of an arbitrary nonempty set $Q \subset X$ is nonempty and, moreover, for any convex subset $S$ of $Q$ there exists a convex component $C \subset \sigma^+(Q)$ such that $S \subset C.$

In addition, the family $\sigma^+(Q)$ of all convex components of a set $Q$ is a covering of $Q,$ i.e., the equality
\vspace{-4mm}
\begin{equation}\label{e1.1}
    Q = \bigcup\{C\,|\,C \in \sigma^+(Q)\}
\end{equation}
\vspace{-4mm}
holds.
\end{theorem}
\vspace{1mm}
The equality \eqref{e1.1} follows from the fact that each one-point subset of $Q$ is convex and, consequently, is contained in some convex component of $Q.$

\begin{remark}\label{rem1.1}
{\rm It is easily seen that, when a set $Q$ is a cone (this means that $\lambda x \in Q$ for all $x \in Q$ and $\lambda > 0),$ any its convex component is also a cone.}
\end{remark}

\begin{remark}\label{rem1.2}
{\rm If $X$ is a Hausdorff topological vector space and $Q$ is a closed subset of $X,$ then any convex component of $Q$ is closed as well.}
\end{remark}

To the best of my knowledge for the first time the families of maximal convex subsets were used for global analysis of sets by Valentine in his 1930 monograph \cite{Valent}. The equality \eqref{e1.1} was established by Smith C.R. in the small note \cite{Smith}. The term `convex component' was introduced by Gorokhovik and Zorko in the paper \cite{GorZor} devoted to nonconvex polyhedral sets.

By \textit{a convex complement} of a set $Q \subset X$ we call a maximal (in the sense of inclusion) convex subsets $D \subset X$ such that $D\bigcap Q = \varnothing.$

It is evident that the notion of a convex complement is related to the notion of a convex  component: any convex complement of $Q$ is a convex component of $X \setminus Q$ and vice versa.

The family of all convex complements of a set $Q$ will be denoted by $\sigma^-(Q).$

\begin{theorem}\label{th1.2}
The family of convex complements $\sigma^-(Q)$ of an arbitrary set $Q \subset X$ with $X \setminus Q \ne \varnothing$ is nonempty and, moreover,
\begin{equation}\label{e1.2}
    Q = \bigcap\{X \setminus D\,|\,D \in \sigma^-(Q)\}.
\end{equation}
Besides, for any convex subset $S \subset X \setminus Q$  there exists a convex complement $D \in \sigma^-(Q)$ such that $S \subset D.$
\end{theorem}

We demonstrate next that the recession cone of a set coincides with the intersection of the recession cones of all convex components of this set.

Recall \cite{Rock} that a vector $y \in X$ defines a recession direction for a subset $Q \subset X$ if $x + ty \in Q$ for all $x \in Q$ and all $t \in [0,+\infty).$

The collection of all vectors which define recession directions for a set $Q$ will be denoted by $Q^\infty.$ It is easy to check that $Q^\infty$ is a convex cone with $0_X \in Q^\infty.$ Besides, for any proper subset $Q \subset X$ the equality  $Q^{\infty} = -\ (X \setminus Q)^{\infty}$ holds. Extending this equality to improper subsets we assume that ${\varnothing}^{\infty} = X$ and $X^{\infty} = {\varnothing}.$

\begin{theorem}\label{th1.3}
For any nonempty subset $Q$ of a real vector space $X$ the equalities $$Q^{\infty} = \bigcap\{\ S^{\infty}\,|\, S \in \sigma^+(Q)\ \} = -\bigcap\{\ C^{\infty}\,|\, C \in \sigma^-(Q)\ \}$$
hold.
\end{theorem}

\begin{proof} Let $y \in Q^{\infty}$ and $S \in \sigma^+(Q).$ The set
$S_1 := \{z = x + ty \ |\ x \in S, t \in [0,+\infty) \}$ is a convex subset of $Q$ and, in addition, $S \subset S_1.$ Since $S$ is a maximal convex subset of $Q,$ we conclude that $S = S_1$ and, hence, $x + ty
\in S$ for all $x \in S$ and all $t \in [0,+\infty).$ It proves that $y \in S^{\infty}$ for all $S \in \sigma^+(Q).$

Conversely, if $y \in \bigcap\{\ S^{\infty} \ \mid \ S
\in \sigma^+(Q)\ \},$ it follows from the equality $Q = \bigcup\{\ S \ \mid \ S \in \sigma^+(Q)\ \},$ that $x + ty \in Q$ for all $x
\in Q$ and all $t \in [0,+\infty),$ i.e., $y \in Q^{\infty}.$ Thus, the proof of the first equality $Q^{\infty} = \bigcap\{\ S^{\infty}\,|\, S \in \sigma^+(Q)\ \}$ is complete.

Then the second equality is immediate from the equality $Q^{\infty} = -\ (X \setminus Q)^{\infty}$.
\end{proof}

Before compliting this section, we recall \cite{Rock,HULem1} that when $Q$ is a closed convex subset of a Hausdorff topological vector space, $h \in Q^\infty$ if and only if for some fixed point $\bar{x} \in Q$ one has $\bar{x} + th \in Q$ for all $t \ge 0.$

\section{Minimal convex majorants and maximal concave minorants of functions} \label{sec2}

Let $X$ be a real vector space, and $f:X \to \overline{{\mathbb R}}$ a function, defined on $X$ and taking values in the extended real line $\overline{{\mathbb R}}:={\mathbb R}\cup\{\pm\infty\}.$

The sets
${\rm epi\hspace{2pt}}f := \{(x,\,\alpha) \in X \times {\mathbb R}\,|\, f(x) \le \alpha\}$ and ${\rm hypo\hspace{2pt}}f := \{(x,\,\alpha) \in X \times {\mathbb R}\,|\, f(x) \ge \alpha\}$ are called, respectively,  \textit{the epigraph} and \textit{the hypograph} of the function~$f;$ the set ${\rm dom}f := \{x \in X \mid |f(x)| < +\infty\}$ is called \textit{the effective domain} of $f.$

A function $f:X \to \overline{{\mathbb R}}$ will be called \textit{$l$-proper}, if $f(x) > -\infty$ for all $x \in X$ and its epigraph ${\rm epi}\hspace{2pt}f$
is a nonempty subset of $X\times {\mathbb R}.$   In the case when $f(x) < +\infty$ for all $x \in X$ and the hypograph ${\rm hypo}\hspace{2pt}f$ of $f$ is a nonempty subset of $X\times {\mathbb R},$  the function $f:X \to \overline{{\mathbb R}}$ will be called \textit{$u$-proper}.

A function $g:X \to {\overline{\mathbb R}}$ is called \textit{convex}, if it is $l$-proper and
$$
g(\lambda x +(1-\lambda)y)\le \lambda g(x) + (1-\lambda)g(y)\,\,\text{for all}\,\,x,\,y \in X\,\,\text{and all}\,\,\lambda \in [0,\,1],
$$
or, equivalently, $g$ is convex, if $g(x) > -\infty$ for all $x \in X$  and its epigraph ${\rm epi}\hspace{2pt}g$  is a nonempty convex subset of $X\times {\mathbb R}.$

A function $h:X \to {\overline{\mathbb R}}$ is called \textit{concave}, if $-h$ is a convex function, or, equivalently, if $h$ is $u$-proper and its hypograph ${\rm hypo}\hspace{2pt}h$
is a nonempty convex subset of $X\times {\mathbb R}.$

\begin{lemma}\label{l3.1}
Any convex subset $G \subset X \times {\mathbb R}$ such that for every point $x \in X$ the subset of reals $\{\gamma \in {\mathbb R}\,|\,(x,\,\gamma) \in G\}$ is either the empty set or a nonempty infinite half-interval bounded from below and unbounded from above defines on $X$ the convex function $g_G:x \to g_G(x):= \inf\{\gamma \in {\mathbb R}\,|\,(x,\,\gamma) \in G\}$ $($under the convention $\inf\varnothing = +\infty).$
\end{lemma}

\proof  It follows immediately from the properties of the set $G$ that the function $g_G$ is $l$-proper. Now, we need to prove that its epigraph ${\rm epi}g_G$ is a convex subset.

Let $(x_1,\gamma_1), (x_2,\gamma_2) \in {\rm epi}g_G.$ Then, for any natural number $n \in {\mathbb N}$ the points $\left(x_1,\gamma_1+ \displaystyle\frac{1}{n}\right),$ $\left(x_2,\gamma_2 + \displaystyle\frac{1}{n}\right)$ also belong to the set $G.$
It implies through the convexity of $G$ that for all $\lambda \in [0,1]$ one has $$\left(\lambda x_1+ (1-\lambda)x_2,\lambda \gamma_1 + (1-\lambda)\gamma_2 +  \displaystyle\frac{1}{n}\right) \in G\,\,\forall\,\,n \in {\mathbb N}$$ and
consequently $$g_G(\lambda x_1+ (1-\lambda)x_2) \le \lambda \gamma_1 + (1-\lambda)\gamma_2 + \displaystyle\frac{1}{n}\,\,\forall\,\,n \in {\mathbb N}.$$ Passing to the limit in the latter inequality as $n \to \infty,$  we get $g_G(\lambda x_1+ (1-\lambda)x_2) \le \lambda \gamma_1 + (1-\lambda)\gamma_2,$ which implies that $\lambda (x_1,\gamma_1)+ (1-\lambda)(x_2,\gamma_2) \in  {\rm epi}g_G.$

Thus, the epigraph of the function $g_G$ is a convex set and consequently the function $g_G$ is convex. \hfill $\square$

Note that in general the epigraph of the function $g_G$ does not coincide with the set $G$ because $G$ may not contain some points of the graph of this function, i.~e., some points of the set $\{(x,\,g_G(x))\,|\,x \in {\rm dom}g_G\}.$ However, the equality ${\rm epi}g_G = G \cup \{(x,\,g_G(x))\,|\,x \in {\rm dom}g_G\}$ holds.

A (convex) function $\phi:X \to~{\overline{\mathbb R}}$ is called \textit{a (convex) majorant} of a function \linebreak $f:X \to~{\overline{\mathbb R}},$ if ${\rm epi}\hspace{2pt}\phi \subset {\rm epi}\hspace{2pt}f$ or, equivalently, if $f(x) \le \phi(x)$ for all $x \in X.$

{\it By a minimal convex majorant} of a function $f:X \to {\overline{\mathbb R}}$ we call such its convex majorant $g:X \to {\overline{\mathbb R}}$, which is minimal (in the sense of pointwise ordering of functions defined on $X$ and taking values in ${\overline{\mathbb R}}$) in the family of all convex majorants of the function $f,$ i.~e., such convex majorant $g$ of the function $f,$ for which there is no other convex majorant $\phi$ of the same function $f,$ that differs from $g$ and satisfies the inequality
$\phi(x) \le g(x)$ for all $x \in X.$

\begin{theorem}\label{th3.2}
Let $X$ be a real vector space.
A convex function $g:X \to \overline{\mathbb R}$ is a minimal convex majorant of an $l$-proper function $f:X \to \overline{\mathbb R}$ if and only if its epigraph ${\rm epi\hspace{2pt}}g$ is a convex component of the epigraph ${\rm epi\hspace{2pt}}f$ of the function $f.$
\par The family $\Sigma^+(f)$ of all minimal convex majorants of any  $l$-proper function \linebreak $f:X \to \overline{\mathbb R}$ is nonempty and, furthermore, for each convex majorant $q$ of the function $f$ there exists a minimal convex majorant $g \in \Sigma^+(f)$ such that $g(x) \le q(x)$ for all $x \in X$

Moreover, the function $f$ is represented in the form
\begin{equation}\label{e3.1}
f(x) = \min\limits_{g \in \Sigma^+(f)}g(x)\,\,\text{для всех}\,\,x \in X.
\end{equation}
\end{theorem}

\proof The sufficiency of the first assertion is immediate from the definitions. Prove the necessity.

Let $g:X \to \overline{\mathbb R}$ be a minimal convex majorant of an  $l$-proper function $f:X \to \overline{\mathbb R}.$ Because ${\rm epi\hspace{2pt}}g$ is a convex set and ${\rm epi\hspace{2pt}}g \subset {\rm epi\hspace{2pt}}f$, then through Theorem~\ref{th1.1} in the family $\sigma^+({\rm epi\hspace{2pt}}f),$ consisting of all convex components of the epigraph of the function $f,$ there is a convex component $T$  such that ${\rm epi\hspace{2pt}}g \subset T.$
Since the vector $(0_X,\,1) \in X\times {\mathbb R}$ ($0_X$ is the origin of the space $X$) belongs to the recession cone $({\rm epi\hspace{2pt}}f)^{\infty}$ of the epigraph of $f,$ we conclude through Theorem~\ref{th1.2} that $(0_X,\,1) \in T^{\infty}.$ From this property of the set $T$ and the facts that the function $f$ is $l$-proper and
$T \subset {\rm epi\hspace{2pt}}f$ it follows then that for each point $x \in X$ the subset of reals $\{\gamma \in {\mathbb R}\,|\,(x,\,\gamma) \in T\}$ is either the empty set or a nonempty infinite half-interval bounded from below and unbounded from above.
Hence, through Lemma~\ref{l3.1} a convex component $T$ defines the convex function $g_T:x \to g_T(x):= \inf\{\gamma\,|\,(x,\,\gamma) \in T\}$  with $T$ being a subset of ${\rm epi}g_T.$ Now it follows immediately from the definition of the function $g_T$ and the inclusion  $T \subset {\rm epi\hspace{2pt}}f$ that  $f(x) \le g_T(x)$ for all $x \in X,$ i.~e., the function $g_T$ is a convex majorant of the function  $f,$ therefore  ${\rm epi\hspace{2pt}}g_T \subset {\rm epi\hspace{2pt}}f.$ Since $T$ is a convex component of the epigraph of the function $f$, it follows from the inclusions $T \subset {\rm epi\hspace{2pt}}g_T \subset {\rm epi\hspace{2pt}}f$ and the convexity of ${\rm epi\hspace{2pt}}g_T$ that ${\rm epi\hspace{2pt}}g_T = T.$  On the other hand, from  the inclusion ${\rm epi\hspace{2pt}}g \subset T = {\rm epi\hspace{2pt}}g_T$ and the fact that $g$ is a minimal convex majorant of the function $f,$ we conclude that $g = g_T$ and, consequently, ${\rm epi\hspace{2pt}}g = T.$
Thus, the first assertion of the theorem is proved.

The second assertion is then justified  through Theorem \ref{th1.1}.

To prove the equality \eqref{e3.1} we use the fact that through Theorem \ref{th1.1} the family of convex components $\sigma^+({\rm epi\hspace{2pt}}f)$ is a covering of the epigraph ${\rm epi\hspace{2pt}}f.$ From this we get that for each $x \in X$ with $f(x) < +\infty$ there exists a convex component $T_x \in \sigma^+({\rm epi}\hspace{2pt}f)$ which contains the point $(x,\,f(x)).$ Hence, for every point $x \in {\rm dom}f$ the equality $f(x) =g_{T_x}(x)$ holds and, consequently, $f(x) = \min\limits_{T \in \sigma^+({\rm epi\hspace{2pt}}f)}g_T(x) = \min\limits_{g \in \Sigma^+(f)}g.$
If the function $f$ takes the value $+\infty$ at a point $x \in X,$ then $g(x) = +\infty$ for all $g \in \Sigma^+(f)$ and, consequently, for such points the equality \eqref{e3.1} also holds.
\hfill $\Box$

The notion of a maximal concave minorant of a function is defined symmetrically to a minimal convex majorant.

A (concave) function $\omega:X \to {\overline{\mathbb{R}}}$ is called \textit{a (concave) minorant} of a function $f:X \to {\overline{\mathbb{R}}}$ if ${\rm hypo\hspace{2pt}}\omega \subset {\rm hypo\hspace{2pt}}f$ or, equivalently, if $\omega(x) \le f(x)$ for all $x \in X.$

By \textit{a maximal concave minorant} of a $f:X \to {\overline{\mathbb R}}$ we call  such its concave minorant $h:X \to {\overline{\mathbb R}}$, that is maximal (in the sense of the pointwise ordering) in the family of all concave minorants of the function $f,$  i.~e., such concave minorant $h$ of the function $f,$ for which there exists no other concave minorant $w$ of the function $f,$ that is different from $h$ and satisfies the inequality
$w(x) \ge h(x)$ for all $x \in X.$

A counterpart of Theorem \ref{th3.2} for maximal concave minorants is formulated as follows.

\begin{theorem}\label{th3.3}
Let $X$ be a real vector space.
A concave function $h:X \to \overline{\mathbb R}$ is a maximal concave minorant of an $u$-proper function $f:X \to \overline{\mathbb R}$ if and only if its hypograph ${\rm hypo\hspace{2pt}}g$ is a convex component of the hypograph ${\rm hypo\hspace{2pt}}f$ of the function~$f.$
\par The family $\Sigma^-(f)$ of all maximal concave minorants of any  $u$-proper function \linebreak $f:X \to \overline{\mathbb R}$ is nonempty and, furthermore, for each concave minorant $w$ of the function $f$ there exists a maximal convex minorant $h \in \Sigma^-(f)$ such that $w(x) \le h(x)$ for all $x \in X.$

Moreover, the function $f$ is represented in the form
\begin{equation}\label{e3.1a}
f(x) = \max\limits_{h \in \Sigma^-(f)}h(x)\,\,\text{для всех}\,\,x \in X.
\end{equation}
\end{theorem}

Before completing this section we will discuss shortly some peculiarities  of minimal convex majorants and maximal concave minorants corresponding positively homogeneous functions.

Recall, that an extended-real-valued function  $p:X \to {\overline{\mathbb R}}$ defined on a real vector space $X$ is called \textit{positively homogeneous} (for short, p.h.), if
\begin{equation}\label{e4.6}
p(\lambda x)=\lambda p(x)\,\,\text{for all}\,\,x \in X\,\,\text{and all}\,\,\lambda > 0.
\end{equation}
or, equivalently, if its epigraph ${\rm epi\hspace{2pt}}p$
is a cone in $X \times {\mathbb R}.$


A convex (respectively, concave) p.h. function $p:X \to {\overline{\mathbb R}}$ is called \textit{sublinear} (respectively, \textit{superlinear}).

Recall that convex and, consequently, sublinear functions are supposed to be $l$-proper, while superlinear ones are $u$-proper.

A rather comprehensive overview of properties of various classes of p.h. functions is contained in \cite{GorTr2016}.

\begin{theorem}\label{th3.7}
Let $X$ be a real vector space. Each minimal convex majorant of $l$-proper p.h. function defined on $X$ is sublinear, whereas each maximal concave minorant of $u$-proper p.h. function defined on $X$ is superlinear.
\end{theorem}

\proof Prove the claim only for minimal convex majorants of a $l$-proper p.h. function. Since the epigraph of a $l$-proper p.h. function $p$ is a cone in $X \times {\mathbb R}$, that does not contain vertical lines, each convex component of ${\rm epi\hspace{2pt}}p$ is a convex cone not containing vertical lines.  Consequently, through Theorem~\ref{th3.2} minimal convex majorants of $p$ are convex p.h. functions, i.e., sublinear ones. \hfill $\square$

It follows from Theorems \ref{th3.2} and \ref{th3.7} that \textit{for each $l$-proper p.h. function $p:X \to \overline{{\mathbb R}}$ the family $S^+(p)$ of all minimal sublinear majorants of $p$ is nonempty and $p$ is represented in the form}
\begin{equation*}\label{e4.2a}
p(x)=\min\limits_{\varphi \in S^+(p)}\varphi(x)\,\,\text{for all}\,\,x \in X.
\end{equation*}

Note that Castellani in \cite{Cast1,Cast2} proved a related result: each $l$-proper p.h. function $p:X \to \overline{{\mathbb R}}$ defined on a Banach space $X$ can be represented in the form
\begin{equation*}\label{e4.2b}
p(x)=\min\limits_{\varphi \in \Phi}\varphi(x)\,\,\text{for all}\,\,x \in X,
\end{equation*}
where $\Phi$ is a family of extended-real-valued sublinear functions.

To prove this result Castellani used arguments not connected with minimal sublinear majorants.

\section{The lower envelope presentation of real-valued functions by subfamilies of minimal convex majorants that are Lipschitz continuous}
\label{sec3}

Note, that even in the case when the function $f:X \to \overline{\mathbb R}$ takes only finite values for all $x \in X$ among its minimal convex majorants one can find those that take the value $+\infty$ on some parts of $X.$  For example, the convex functions $$g_1(x_1,x_2) = \begin{cases} 0, & \text{when $x_1=0,$}\\ +\infty, & \text{when $x_1 \ne 0,$}\end{cases} \,\,\,\,\text{and}\,\,\,\,g_2(x_1,x_2) = \begin{cases} 0, & \text{when $x_2=0,$}\\ +\infty, & \text{when $x_2 \ne 0,$}\end{cases}$$ are minimal convex majorants of the function $f(x_1,x_2) = \sqrt{|x_1x_2|}.$  It is easily verified that if we remove these majorants from $\Sigma^+(f),$ then the equality \eqref{e3.1} will not hold for $f$ on the lines $x_1 = 0$ and $x_2 = 0.$

Our main aim now is to establish characteristic properties of those real-valued functions defined on a real normed space $X$ that admit a lower (upper) envelope representation by such their
minimal convex majorants (maximal concave minorants) which take only finite real values and, moreover, are Lipschitz continuous on the whole space~$X.$

\begin{theorem}\label{th4.1}
Let $X$ be a normed vector space. A real-valued function $f:X \to {\mathbb R}$ is Lipschitz continuous on the whole space $X$ with the Lipschitz constant $k >0$ if and only if each minimal convex majorant (equivalently, each maximal concave minorant) of $f$ is also real-valued and Lipschitz continuous on $X$ with Lipschitz constant not exceeding $k.$
\end{theorem}

\proof It immediately follows from the definition that a function $f: X \to {\mathbb{R}}$ is Lipschitz continuous on $X$ with the Lipschitz constant $k>0$ if and only if the convex cone $E_k:=\{(x,\alpha) \in X \times {\mathbb{R}} \mid k\|x\| \le \alpha\}$ is contained into the recession cone $({\rm epi\hspace{2pt}}f)^\infty$ of its epigraph. Since through Theorem \ref{th1.1} $({\rm epi\hspace{2pt}}f)^\infty = \bigcap\{T^\infty \mid T \in \sigma^+({\rm epi\hspace{2pt}}f)\},$ the inclusion $E_k \subset ({\rm epi\hspace{2pt}}f)^\infty$ is equivalent to the condition $E_k \subset T^\infty$ for all $T \in \sigma^+({\rm epi\hspace{2pt}}f).$ Due to Theorem \ref{th3.2} the family $\sigma^+({\rm epi\hspace{2pt}}f)$ of convex components of the epigraph ${\rm epi\hspace{2pt}}f$ coincides with the family $\{{\rm epi\hspace{2pt}}g \mid g \in \Sigma^+(f)\},$ where $\Sigma^+(f)$ is the family of minimal convex majorants of the function $f.$ Consequently, the Lipschitz continuity of $f$ is equivalent to the condition $E_k \subset ({\rm epi\hspace{2pt}}g)^\infty$ for all $g \in \Sigma^+(f)$ and this in turn is equivalent to that each minimal convex majorant $g$ of the function $f$ is Lipschitz continuous on $X$ with its Lipschitz constant being not grater then~$k.$
\hfill $\square$

Further we show that replacing in \eqref{e3.1} the minimum by the infimum we can enlarged the class of functions which can be represent as the lower envelope of minimal convex majorants which are Lipschitz continuous on $X.$ We will denote the subfamily of minimal convex majorants of a function $f$ which are Lipschitz continuous on the whole space $X$  by the symbol $\Sigma_{Lip}^+(f).$

Let us begin with some preliminaries.

Let $X$ be a metric space with a distance function $d:X\times X \to {\mathbb R}$ and let $f:X \to {\mathbb R}$ be a real-valued function defined on $X.$ For any real $k > 0$ the function $f_k:x \to f_k(x)$ with
\begin{equation}\label{e4.2}
f_k(x) := \sup\limits_{y \in X}(f(y) - kd(x,y))\,\,\text{for all}\,\,x \in X
\end{equation}
is called \cite{Bauschke,RW98} \textit{the Pasch-Hausdorff envelope} of $f$ for the value $k.$

It follows from the inequality
\begin{equation}\label{e4.2d}
f(y) - kd(y,x) \ge f(y) - kd(y,\bar{x}) - kd(\bar{x},x)\,\,\text{for all}\,\,x,\bar{x} \in X\,\,\text{and all}\,\, k > 0,
\end{equation}
that for all $k > 0$ either $f_k(x) \equiv +\infty$ or $f_k(x) < +\infty$ for all $x \in X.$

In the next proposition we summarize the main properties of the Pasch-Hausdorff envelopes and provide them with short proofs (for more details and historical comments we refer to \cite{Bauschke,RW98,HU,GT2015}).

\begin{proposition}\label{pr4.1}
Let $f:X \to {\mathbb R}$ be a real-valued function defined on a metric space $X$ and let $k >0.$ When $f_{k}(x) < +\infty$  for all $x \in X,$ then $f_k$ is Lipschitz continuous on $X$ with Lipschitz constant $k$ and, moreover, $f_k$ is the least of all majorants of $f$ that are Lipschitz continuous on $X$ with Lipschitz constant $k.$

Furthermore, if there exists $\bar{k} > 0$ such that $f_{\bar{k}}(x) < +\infty$  for all $x \in X,$ then for each $x \in X$ the function $k \to f_k(x)$ is nonincreasing over the interval $[\bar{k}, +\infty)$ and
$$
({\rm cl}f)^\uparrow(x) = \inf\limits_{k \ge \bar{k}}f_k(x)\,\,\text{for all}\,\,x \in X.
$$
Here $({\rm cl}f)^\uparrow$ stands for the upper semicontinuous closure of $f$.
\end{proposition}

\proof It is easy to see from \eqref{e4.2d} that, provided $f_k(x) < +\infty\,\forall\,x \in X,$  the function $f_k$ is Lipschitz continuous on $X$ with Lipschitz constant $k.$ Besides, letting $y=x$ in \eqref{e4.2}, we get $f(x) \le f_k(x)\,\,\forall\,\,x \in X.$ Thus, $f$ is majorized by $f_k$ with $f_k$ being Lipschitz continuous on $X$ with Lipschitz constant $k.$ Now, consider an arbitrary function $g:X \to {\mathbb R}$ that is  Lipschitz continuous on $X$ with Lipschitz constant $k$ and such that $f(x) \le g(x)$ for all $x \in X.$ Since $g(x) \ge g(y) - k\|y - x\| \ge f(y) - k\|y - x\|$ for all $x,y \in X,$ we have $g(x) \ge f_k(x)$ for all $x \in X.$ It proves that $f_k$ is the least of all Lipschitz continuous functions with Lipschitz constant $k$ which majorize~$f.$

Suppose that for $f$ there exists $\bar{k} > 0$ such that $f_{\bar{k}}(x) < +\infty$ for all $x \in X$ and take $k_1,k_2$ satisfying $\bar{k} \le k_1 \le k_2.$ It is easy to get from the inequality
$$
f(y) - k_2\|y - x\| \le f(y) - k_1\|y - x\|\,\,\text{for all}\,\,x,y \in X
$$
that $f_{k_2}(x) \le f_{k_1}(x)\,\,\forall\,\,x \in X.$

Thus, for each $x \in X$ the function $k \to f_k(x)$ is nonincreasing on $[\bar{k},+\infty)$ and, consequently, $\lim\limits_{k \to +\infty}f_k(x) = \inf\limits_{k \ge \bar{k}}f_k(x)$ for all $x \in X.$

Since $f(x) \le f_k(x)\,\,\forall\,\,x \in X,$ we have that $f(x) \le \inf\limits_{k \ge \bar{k}}f_k(x)\,\,\forall\,\,x \in X.$ Because the function $x \to \inf\limits_{k \ge \bar{k}}f_k(x)$ is upper semicontinuous on $X$ we have also that $({\rm cl}f)^\uparrow(x) \le \inf\limits_{k \ge \bar{k}}f_k(x)\,\,\forall\,\,x \in X.$ Show that the opposite inequality $({\rm cl}f)^\uparrow(x) \ge \inf\limits_{k \ge \bar{k}}f_k(x)\,\,\forall\,\,x \in X$ holds as well and thereby prove the equality $({\rm cl}f)^\uparrow(x)=\inf\limits_{k> 0}f_k(x)\,\,\forall\,\,x \in X.$

Take an arbitrary point $\bar{x} \in X$ and choose, for each positive real $k \ge \bar{k},$ a point $y_k \in X$ such that $$f(y_k)-kd(y_k,\bar{x}) \ge f_k(\bar{x}) -\displaystyle\frac{1}{k}.$$

From the last inequality, since $f_{\bar{k}}(\bar{x}) + \bar{k}d(x,\bar{x}) \ge f_{\bar{k}}(x) \ge f(x)\,\,\forall\,\,x \in X,$  we get for $k \ge \bar{k}$ that
$$
kd(y_k,\bar{x}) \le f_{\bar{k}}(y_k) - f_k(\bar{x}) + \displaystyle\frac{1}{k} \le f_{\bar{k}}(\bar{x}) +\bar{k}d(y_k,\bar{x}) - f(\bar{x})+1,
$$
which implies  $(k - \bar{k})d(y_k,\bar{x}) \le f_{\bar{k}}(\bar{x}) - f(\bar{x})+1.$
Consequently,
$\lim\limits_{k \to \infty}d(y_k,\,\bar{x}) = 0$ and, hence, $\lim\limits_{k \to \infty}y_k = \bar{x}.$

Using the inequality
$$f(y_k) \ge f(y_k)-kd(y_k,\bar{x}) \ge f_k(\bar{x}) - \displaystyle\frac{1}{k},$$
we obtain
$$
({\rm cl}f)^\uparrow(\bar{x}) = \displaystyle\limsup_{y \to \bar{x}}f(y) \ge \limsup\limits_{k \to +\infty}f(y_k) \ge \lim\limits_{k \to \infty}f_k(\bar{x}) -\lim\limits_{k \to \infty}\displaystyle\frac{1}{k} = \inf\limits_{k \ge \bar{k}}f_k(\bar{x}).
$$
Due to the arbitrary choice of $\bar{x} \in X,$ we get $({\rm cl}f)^\uparrow(x) \ge \inf\limits_{k \ge \bar{k}} f_k(x)\,\,\forall\,\,x \in X.$

Thus, $({\rm cl}f)^\uparrow(x) = \inf\limits_{k \ge \bar{k}} f_k(x)\,\,\forall\,\,x \in X.$ \hfill $\Box$

The assumption that for some $\bar{k} > 0$ the Pasch-Hausdorff envelope of a function $f$ for the value $\bar{k}$ takes finite values for all $x \in X,$ i.e., $f_{\bar{k}}(x) < +\infty\,\,\forall\,\,x \in X,$ is equivalent to the fact that the function $f$  is majorized by a function that is Lipschitz continuous on $X$ with Lipschitz constant $\bar{k}$. When this is the case for a function $f$ we will say that $f$ is \textit{$\bar{k}$-Lipschitz bounded from above} on $X.$  We will simply say that $f$ is \textit{Lipschitz bounded from above} if it is $k$-Lipschitz bounded from above on $X$ for some $k$. Symmetrically, a function $f$ is \textit{Lipschitz bounded from below} if $-f$ is Lipschitz bounded from above. At last, a function $f$ is called \textit{Lipschitz bounded} if it is Lipschitz bounded both from below and from above.

\begin{proposition}\label{pr4.2}
Let $p:X \to {\mathbb R}$ be a p.h. function defined on a real normed vector space $X.$
The Pasch-Hausdorff envelope $p_k$ of a p.h. function $p$ for any value $k$ is also a p.h. function. A p.h. function $p:X \to {\mathbb R}$ is upper semicontinuous on $X$ if and only if there exists $\bar{k} > 0$ such that $p(x) \le \bar{k}\|x\|$ for all $x \in X$ and $p$ is represented in the form
\begin{equation}\label{e4.1h}
p(x) = \inf\limits_{k \ge \bar{k}}p_k(x)\,\,\text{for all}\,\,x \in X,
\end{equation}
where $p_k$ is the Pasch-Hausdorff envelope of $p$ for the value $k.$
\end{proposition}

\proof If $p$ is a p.h. function, then for all $x \in X$ and all $\lambda > 0$ we have
$$
p_k(\lambda x) = \sup\limits_{y \in X}(p(y)-k\|y-\lambda x\|)=\lambda \sup\limits_{y \in X}(p(\lambda^{-1}y) -k\|\lambda^{-1}y -x\|)= \lambda p_k(x).
$$
Thus, $p_k$ is also a p.h. function.

To prove the second assertion we note first that, if a p.h. function $p:X \to {\mathbb R}$ is upper semicontinuous on $X,$ it is upper semicontinuous at $x=0$ and, consequently, for any $\varepsilon > 0$ we can find $\delta > 0$ such that $p(x) \le p(0) + \varepsilon = \varepsilon$ for all $x \in B_\delta(0).$ It implies through positively homogeneity of $p$ that $p(x) \le \displaystyle\frac{\varepsilon}{\delta}\|x\|$ for all $ x \in X.$ Hence, an arbitrary upper semicontinuous function $p$ is majorized by a p.h. function that is Lipschitz continuous on $X$ and, consequently, there exists $\bar{k} > 0$ such that $p_{\bar{k}}(x) < +\infty$ for all $x \in X.$

To complete the proof of the necessary part of the second assertion, it is sufficient to apply the results of Proposition \ref{pr4.1}.

The sufficient part is straightforward.
\hfill $\square$

To prove the next theorem we need the following lemma.

\begin{lemma}\label{l4.1}
Let $g:X \to {\mathbb R}$ and $\tilde{g}:X \to {\mathbb R}$ be real-valued convex functions defined on a normed space $X.$ If $g(x) \le \tilde{g}(x)$ for all $x \in X$ and the function $\tilde{g}$ is Lipschitz continuous on $X$ with a Lipschitz constant $k > 0$ then $g$ is also Lipschitz continuous with the same Lipschitz constant $k.$
\end{lemma}

\proof Take an arbitrary point $\bar{x} \in X.$ Then $g(x) \le \tilde{g}(x) \le \tilde{g}(\bar{x}) + k\|x-\bar{x}\|$ for all $x \in X$ and, consequently, $g$ is bounded from above on an arbitrary ball $B_\delta(\bar{x}):=\{x \in X \mid \|x - \bar{x}\| \le \delta\}.$ It implies that $g$ is continuous at $\bar{x}.$ Due to the arbitrary choice of $\bar{x}$ the function $g$ is continuous on the whole space $X.$ Hence, the epigraph ${\rm epi}g$ is a closed convex subset of the space $X \times {\mathbb R}$ and ${\rm epi}\tilde{g} \subset {\rm epi}g.$ These properties imply $({\rm epi}\tilde{g})^\infty \subset ({\rm epi}g)^\infty.$ Since $\tilde{g}$ is Lipschitz continuous on $X$ with a Lipschitz constant $k > 0,$ then $E_k:=\{(x,\alpha) \in X \times {\mathbb R} \mid k\|x\| \le \alpha\} \subset   ({\rm epi}\tilde{g})^\infty$ and, consequently, $E_k \subset ({\rm epi}g)^\infty.$ It means that $g$ is Lipschitz continuous on $X$ with Lipschitz constant~$k.$

\hfill $\Box$

\begin{theorem}\label{th4.4}
Let $f:X \to {\mathbb R}$ be a real-valued functions defined on a normed space $X$ and let $\Sigma_{Lip}^+(f)$ be the subfamily  of all minimal convex majorants of $f$ that are Lipschitz continuous on $X.$ Then, $\Sigma_{Lip}^+(f) \ne \varnothing$ if and only if $f$ is Lipschitz bounded from above.
\end{theorem}

\proof Prove the ``only if'' part. Let $g: X \to {\mathbb R}$ be a Lipschitz continuous function such that $f(x) \le g(x)$ for all $x \in X.$ Through Theorem \ref{th4.1} we have that $\Sigma_{Lip}^+(g) =\Sigma^+(g) \ne \varnothing.$ By Theorem \ref{th3.2} for each  $\tilde{g} \in \Sigma_{Lip}^+(g),$ as for a convex majorant of $f,$  there exists $\hat{g} \in \Sigma^+(f)$ such that $\hat{g}(x) \le \tilde{g}(x)$ for all $x \in X.$ Inasmuch as $\tilde{g}$ is Lipschitz continuous on $X,$ through Lemma \ref{l4.1} the function $\hat{g}$ is also Lipschitz continuous on $X$ and, consequently, $\hat{g} \in \Sigma_{Lip}^+(f).$

The proof of the ``if'' part is straightforward. \hfill $\Box$

\begin{theorem}\label{th4.5k}
Let $f:X \to {\mathbb R}$ be a real-valued function defined on a real normed space $X.$
If the function $f$ is $\bar{k}$-Lipschitz bounded from above on $X$, then
\begin{equation}\label{e4.9}
\Sigma_{Lip}^+(f) = \bigcup\limits_{k \ge \bar{k}}\Sigma^+(f_k).
\end{equation}
and, moreover, for all $k_1,k_2$ such that $\bar{k} \le k_1 \le k_2$ one has $\Sigma^+(f_{k_1}) \subseteq \Sigma^+(f_{k_2}).$

Here $f_k$ stands for the Pasch-Hausdorff envelope of $f$ for the value $k.$
\end{theorem}

\proof  Let us prove first that $\Sigma^+(f_{k}) \subseteq \Sigma^+_{Lip}(f)$ for all $k \ge \bar{k}.$

Take an arbitrary $k \ge \bar{k}.$ By Theorems \ref{th1.1} and \ref{th3.2} for each $g \in \Sigma^+(f_k)$ there exists a minimal convex majorant $\tilde{g} \in \Sigma^+(f)$ such that $\tilde{g}(x) \le g(x)$ for all $x \in X.$ Since both the functions, $g$ and $\tilde{g}$, are convex and $g$ is Lipschitz continuous on $X$ with Lipschitz constant $k$, by Lemma \ref{l4.1} $\tilde{g}$ is Lipschitz continuous on $X$ as well with the same Lipschitz constant $k.$ Thus, $\tilde{g}$ is a minimal $k$-Lipschitz continuous convex majorant of $f$ and, consequently, $\tilde{g} \in \Sigma^+_{Lip}(f).$  Since $f_k$ is the least of $k$-Lipschitz continuous majorant of  $f$, we conclude that $f_k(x) \le \tilde{g}(x)\,\,\forall\,\,x \in X.$ Now, taking into account, that $g \in \Sigma^+(f_k)$ and $\tilde{g}(x) \le g(x)\,\,\forall\,\,x \in X$, we get the equality $\tilde{g} = g.$  This proves that $\Sigma^+(f_k) \subseteq \Sigma^+_{Lip}(f).$  Due to an arbitrary choice of $k,\,k \ge \bar{k},$ we get $\bigcup\limits_{k \ge \bar{k}}\Sigma^+(f_k) \subset \Sigma^+_{Lip}(f).$

To prove the converse inclusion, choose $g \in \Sigma_{Lip}^+(f)$ and let  $k$ be a Lipschitz constant of $g.$ Without loss of generality we can suppose that $k \ge \bar{k}.$ The function $g$ is a convex majorant of $f_k$ and, consequently, there exists $\hat{g} \in \Sigma^+(f_k)$ such that $\hat{g}(x) \le g(x)\,\,\forall\,\,x \in X.$ Since $\hat{g}$ is a convex majorant of $f$ and $g \in \Sigma_{Lip}^+(f) \subseteq \Sigma^+(f),$ we get that $g = \hat{g}.$ It implies that $g \in \Sigma^+(f_k).$ Thus, $\Sigma^+_{Lip}(f) \subset \bigcup\limits_{k \ge \bar{k}}\Sigma^+(f_k).$ It completes the proof of the equality \eqref{e4.9}.

For arbitrary $k_1,k_2$ such that $\bar{k} \le k_1 \le k_2$ and any $g_1 \in \Sigma^+(f_{k_1})$ we have $f(x) \le f_{k_2}(x) \le f_{k_1}(x) \le g_1(x)\,\,\forall\,\,x \in X.$ We conclude from these inequalities that for each $g_1 \in \Sigma^+(f_{k_1})$ there exists $g_2 \in \Sigma^+(f_{k_2})$ such that $g_2(x) \le g_1(x)\,\,\forall x \in X.$ Since through the equality \eqref{e4.9} $g_1 \in \Sigma^+_{Lip}(f)$, we get from the last inequality that $g_1 = g_2$ and, consequently, $g_1 \in \Sigma^+(f_{k_2}).$ Thus, $\Sigma^+(f_{k_1}) \subseteq \Sigma^+(f_{k_2}).$ \hfill $\Box$

\begin{theorem}\label{th4.5}
Let $f:X \to {\mathbb R}$ be a real-valued function defined on a real normed space $X.$
If the function $f$ is Lipschitz bounded from above, then the upper semicontinuous closure of $f$ admits the representation
\begin{equation}\label{e4.10}
({\rm cl}f)^\uparrow(x) = \inf\limits_{g \in \Sigma_{Lip}^+(f)}g(x)\,\,\text{for all}\,\,x \in X.
\end{equation}
\end{theorem}

\proof By Proposition \ref{pr4.1} the upper semicontinuous closure of the function $f$ admits the presentation
\begin{equation}\label{e4.4}
({\rm cl}f)^\uparrow (x) = \inf\limits_{k \ge \bar{k}} f_k(x)\,\,\forall\,\,x \in X,
\end{equation}
where $f_k$ is the Pasch-Hausdorff envelope of $f$ for the value $k.$
In turn, by Theorem~\ref{th3.2} we can represent each function $f_k$ as follows
\begin{equation}\label{e4.5}
f_k(x) = \min\limits_{g \in \Sigma^+(f_k)}g(x)\,\,\text{for all}\,\,x \in X,
\end{equation}
where $\Sigma^+(f_k)$ is the family of all minimal convex majorants of the function $f_k.$

Then, it follows from \eqref{e4.4},\eqref{e4.5}, and \eqref{e4.9} that
\begin{equation*}
({\rm cl}f)^\uparrow(x)=\inf\limits_{k \ge \bar{k}}\min\limits_{g \in \Sigma^+(f_k)}g(x)= \inf\limits_{g\, \in \bigcup\limits_{k \ge \bar{k}}\Sigma^+(f_k)}g(x) = \inf\limits_{g \in \Sigma_{Lip}^+(f)}g(x)\,\,\text{for all}\,\,x \in X.
\end{equation*}
\hfill $\Box$

Now we are ready to prove the central result of this section.

\begin{theorem}\label{th4.7}
Let $f:X \to {\mathbb R}$ be a real-valued function defined on a real normed space $X.$

$(i)$ For the function $f$ to be upper semicontinuous and Lipschitz bounded from above on $X$, it is necessary and sufficient that the family $\Sigma_{Lip}^+(f)$ be nonempty and $f$ admit the following lower envelope representation
\begin{equation}\label{e4.3k}
f(x) = \inf\limits_{g \in \Sigma_{Lip}^+(f)}g(x)\,\,\text{for all}\,\,x \in X.
\end{equation}

Here $\Sigma_{Lip}^+(f)$ stands for the subfamily of all minimal convex majorants which are Lipschitz continuous on $X.$

$(ii)$ For the function $f$ to be lower semicontinuous  and Lipschitz bounded from below on $X,$ it is necessary and sufficient that the family $\Sigma_{Lip}^-(f)$ be nonempty and $f$ admit the following upper envelope representation
\begin{equation}\label{e4.4a}
f(x) = \sup\limits_{h \in \Sigma_{Lip}^-(f)}h(x)\,\,\text{for all}\,\,x \in X.
\end{equation}
Here $\Sigma_{Lip}^-(f)$ stands for the subfamily of all maximal concave minorants which are Lipschitz continuous on $X.$
\end{theorem}

\proof Prove the statement $(i).$ The ``if'' part comes directly from the representation \eqref{e4.3k} and continuity of functions of the family $\Sigma_{Lip}^+(f).$

To prove the ``only if'' part we observe that, since $f$ is upper semicontinuous on $X,$ one has $f = ({\rm cl}f)^\uparrow.$ Thus, the equality \eqref{e4.3k} follows immediately from the equality~\eqref{e4.9}.

\hfill $\Box$

For p.h. functions the requirement of Lipschitz boundedness from above in the statement $(i)$ of the last theorem can be omitted, since each upper semicontinuous p.h. function is bounded from above by the Lipschitz continuous function $x \to k\|x\|$ with a suitable number $k > 0.$ Observe also, that each minimal convex majorant of a p.h. function is sublinear and that each continuous sublinear function $\varphi:X \to {\mathbb R}$ is Lipschitz continuous. Consequently, the family $\Sigma^+_{Lip}(p)$ corresponding to a p.h. function $p: X \to {\mathbb R}$ consists exclusively of continuous sublinear functions.
To emphasize these peculiarities  we will denote the family of all minimal continuous sublinear majorants of a p.h. function $p: X \to {\mathbb R}$ by the symbol $S^+_C(p)$ instead of $\Sigma^+_{Lip}(p).$

Taking into account  these observations we get from Theorem \ref{th4.7} the following characterization of upper (lower) semicontinuous p.h. functions.

\begin{theorem}\label{th4.7a}
Let $X$ be a real normed space.

$(i)$ A real-valued p.h. function $p:X \to {\mathbb R}$ is upper semicontinuous on $X$ if and only if $S_{C}^+(p) \ne \varnothing$ and
\begin{equation}\label{e4.7}
p(x) = \inf\limits_{\varphi \in S_{C}^+(p)}\varphi(x)\,\,\text{for all}\,\,x \in X.
\end{equation}
Here $S_{C}^+(p)$ stands for the subfamily of all maximal continuous sublinear majorants of~$p.$

$(ii)$ A real-valued p.h. function $p:X \to {\mathbb R}$ is lower semicontinuous on $X$ if and only if $S_{C}^-(p) \ne \varnothing$ and
\begin{equation}\label{e4.7a}
p(x) = \sup\limits_{\psi \in S_{C}^-(p)}\psi(x)\,\,\text{for all}\,\,x \in X.
\end{equation}
Here $S_{C}^-(p)$ stands for the subfamily of all maximal continuous superlinear minorants of~$p.$
\end{theorem}

It follows from \eqref{e4.7} and \eqref{e4.7a} that in the Demyanov-Rubinov terminology $S_{C}^+(p)$ and $S_{C}^-(p)$ are, respectively, an exhaustive family of upper convex approximations and an exhaustive family of lower concave approximations of a p.h. functions $p.$

\begin{example}\label{ex4.1}
Let $f: X \to {\mathbb R}$ be a continuous and Lipschitz bounded function. Then, by Theorem \ref{th4.4} both the families $\Sigma^-_{Lip}(f)$ and $\Sigma^+_{Lip}(f)$ are nonempty. Suppose that there exist two continuous affine functions $a_1: X \to {\mathbb R}$ and $a_2: X \to {\mathbb R}$ such that $a_1 \in \Sigma^-_{Lip}(f)$ and $a_2 \in \Sigma^+_{Lip}(f).$ Since $a_1(x) \le f(x) \le a_2(x)\,\,\forall\,\,x \in X,$ we conclude that in this case $f = a_1 = a_2.$ Thus, both the families $\Sigma^-_{Lip}(f)$ and $\Sigma^+_{Lip}(f)$ can simultaneously contain continuous affine functions if and only if $f$ itself is a continuous affine one, and, moreover, in this case $\Sigma^-_{Lip}(f) = \Sigma^+_{Lip}(f) = \{f\}.$
\end{example}

\begin{example}\label{ex4.2}
Let $f: X \to {\mathbb R}$ be a lower semicontinuous convex function function. Since $f$ is bounded from above by a continuous affine function (see, for instance, \cite{ET76}), through Theorem \ref{th4.4} the family $\Sigma^-_{Lip}(f)$ is nonempty and, moreover, by the sandwich theorem \cite{BorVan}  each $g \in \Sigma^-_{Lip}(f)$ is a continuous affine function. Using the Fenchel conjugate  function $f^*(x^*) := \sup\limits_{x \in X}(x^*(x) - f(x))$ ($X^*$ is the dual space of continuous linear functions on $X$) \cite{ET76,Rock,HULem1,BorVan}, we can represent the family $\Sigma^-_{Lip}(f)$  as follows: $\Sigma^-_{Lip}(f) = \{x \to x^*(x) -f^*(x^*) \mid x^* \in X^*\}.$ Besides, the second Fenchel conjugate (biconjugate) $f^{**}$ of $f$ determined by the equality $f^{**}(x) := \sup\limits_{x^* \in X^*}(x^*(x) - f^*(x^*))$ for all $x \in X,$ is nothing more than the upper envelope of $\Sigma^-_{Lip}(f),$ that is,  $f^{**}(x) = \sup\limits_{h \in \Sigma^-_{Lip}(f)}h(x),$ $x \in X.$
\end{example}

Thus, the mapping $\Sigma^-_{Lip}: f \rightarrow \Sigma^-_{Lip}(f)$ can be considered as an extension of the Legendre-Fenchel transformation $f \to f^*$ from the class of lower semicontinuous convex functions to the much more larger collections of lower semicontinuous and Lipschitz bounded from below functions.

\section{Demyanov--Rubinov sub(super)differentials based on directional derivatives} \label{sec4}

The families $\Sigma^-_{Lip}(f)$ and $\Sigma^+_{Lip}(f)$ are global characteristics of lower and upper semicontinuous functions. In this section we introduce their local analogs called the Demyanov-Rubinov subdifferential and the Demyanov-Rubinov superdifferential, respectively. To develop these local constructions we follow one of the conventional approaches based on directional derivatives.

Let $f: X \to {\overline{\mathbb R}}$ be an extended real-valued function defined on a real normed space~$X.$

By $f^\downarrow(\bar{x}\,|\,\cdot)$ and $f^\uparrow(\bar{x}\,|\,\cdot)$ we denote, respectively, the lower (radial) directional derivative and the upper (radial) directional derivative of the function $f$ at a point $\bar{x} \in {\rm int}({\rm dom}f).$
Recall that the values of $f^\downarrow(\bar{x}\,|\,\cdot)$ and $f^\uparrow(\bar{x}\,|\,\cdot)$ at $d \in X$ are defined by
$$
f^\downarrow(\bar{x}\,|\,d) = \liminf\limits_{t \to 0+}\frac{f(\bar{x} + td) - f(\bar{x})}{t}
$$
and
$$
f^\uparrow(\bar{x}\,|\,d) = \limsup\limits_{t \to 0+}\frac{f(\bar{x} + td) - f(\bar{x})}{t}.
$$
It is well-known that $f^\downarrow(\bar{x}\,|\,\cdot)$ and $f^\uparrow(\bar{x}\,|\,\cdot)$ are positively homogeneous functions, which can take infinite values in general. But in what follows we will assume that
$f^\downarrow(\bar{x}\,|\,\cdot)$ is bounded from below on the unit ball $B := \{x \in X  \mid \|x\| \le 1\},$ while $f^\uparrow(\bar{x}\,|\,\cdot)$ is bounded from above on $B.$ Since $f^\downarrow(\bar{x}\,|\,d) \le f^\uparrow(\bar{x}\,|\,d)$ for all $d \in X,$ these assumptions imply that both $f^\downarrow(\bar{x}\,|\,\cdot)$ and $f^\uparrow(\bar{x}\,|\,\cdot)$ are bounded on $B$ and, consequently, take  only finite values.

Using the lower and upper directional derivatives $f^\downarrow(\bar{x}\,|\,\cdot)$ and $f^\uparrow(\bar{x}\,|\,\cdot)$, we introduce the following four notions:

1) \textit{the lower Demyanov--Rubinov sub\-dif\-ferential} (or, for short, \textit{the lower \linebreak $DR$-sub\-dif\-ferential)} of $f$ at $\bar{x}$, denoted by $\underline{\partial}^-_{DR}f(\bar{x})$, that is defined by the equality $\underline{\partial}^-_{DR}f(\bar{x}) := S^-_{C}(f^\downarrow(\bar{x}\,|\,\cdot))$;

2) \textit{the lower $DR$-superdifferential of $f$ at $\bar{x}$,} denoted by $\underline{\partial}^+_{DR}f(\bar{x})$, that is defined by the equality $\underline{\partial}^+_{DR}f(\bar{x}) := S^+_{C}(f^\downarrow(\bar{x}\,|\,\cdot))$;

3) \textit{the upper $DR$-subdifferential of $f$ at $\bar{x}$,} denoted by $\bar{\partial}^-_{DR}f(\bar{x})$, that is defined by the equality $\bar{\partial}^-_{DR}f(\bar{x}) := S^-_{C}(f^\uparrow(\bar{x}\,|\,\cdot))$;

4) \textit{the upper $DR$-superdifferential of $f$ at $\bar{x}$,} denoted by $\bar{\partial}^+_{DR}f(\bar{x})$, that is defined by the equality $\bar{\partial}^+_{DR}f(\bar{x}) := S^+_{C}(f^\uparrow(\bar{x}\,|\,\cdot))$.

When a function $f$ is directionally differentiable at $\bar{x}$ (this means that $f^\downarrow(\bar{x}\,|\,h) = f^\uparrow(\bar{x}\,|\,d) =: f'(\bar{x}\,|\,d)$ for all $d \in X$), we refer to  the coinciding sets $\underline{\partial}^-_{DR}f(\bar{x}) = \bar{\partial}^-_{DR}f(\bar{x})$ and $\underline{\partial}^+_{DR}f(\bar{x}) = \bar{\partial}^+_{DR}f(\bar{x})$ as
\textit{the $DR$-subdifferential of a function $f$ at $\bar{x}$}  and \textit{the $DR$-superdifferential of a function $f$ at $\bar{x}$}, and will denote them by ${\partial}^+_{DR}f(\bar{x})$ and ${\partial}^+_{DR}f(\bar{x}),$  respectively.

\vspace{3mm}

Discuss relationships of $DR$-sub(super)fifferentials with some other known notions of subdifferentiability.

\vspace{2mm}

$A.$  Let  $f: X \to {\overline{\mathbb R}}$ be a lower semicontinuous convex function and let $\bar{x} \in {\rm int}({\rm dom}f)$. Then the $DR$-subdifferential ${\partial}^-_{DR}f(\bar{x})$ of the function $f$ at $\bar{x}$  is nonempty and coincides with the Fenchel--Moreau subdifferential $\partial f(\bar{x})$ of $f$ at $\bar{x}$ in the sense of convex analysis. This conclusion follows from the fact that due to the Hahn-Banach theorem each maximal concave minorant of a continuous sublinear function is, in fact, a continuous linear one. As for the $DR$-superdifferential ${\partial}^+_{DR}f(\bar{x})$ of the function $f$ at $\bar{x}$ it is an one-element family consisting only of the continuous sublinear function $f'(\bar{x}\,|\,\cdot).$

\vspace{2mm}

$B.$ Let $X^*$ be the topological dual of $X,$ i.e., $X^*$ is the space of continuous linear functionals on $X.$ Suppose, that for a function $f:X \to \overline{\mathbb R}$ and a point $\bar{x} \in {\rm int}({\rm dom}f)$ both  the $DR$-subdifferential ${\partial}^-_{DR}f(\bar{x})$ and the $DR$-superdifferential ${\partial}^+_{DR}f(\bar{x})$ of the function $f$ at $\bar{x}$  are nonempty and there exist $x_1^*, x^*_2 \in X^*$ such that $x_1^* \in {\partial}^-_{DR}f(\bar{x})$ and $x_2^* \in {\partial}^+_{DR}f(\bar{x}).$ Then $x_1^*(d) \le f'(\bar{x}\,|\,d) \le x_2^*(d)$ for all $d \in X$ and, consequently, $f'(\bar{x}\,|\,d) = x_1^*(d) = x_2^*(d)$ for all $d \in X.$ We conclude from  this that both the $DR$-subdifferential ${\partial}^-_{DR}f(\bar{x})$ and the $DR$-superdifferential ${\partial}^+_{DR}f(\bar{x})$ contain continuous linear functionals if and only if the directional derivative $f'(\bar{x}\,|\,\cdot)$ of the function $f$ at the point $\bar{x}$ itself is a continuous linear functional, i.e., if and only if $f$ is G\^{a}teaux differentiable at $\bar{x}$. Thus, if a function $f$ is not G\^{a}teaux differentiable at $\bar{x}$ only one of the families, either ${\partial}^-_{DR}f(\bar{x})$ or ${\partial}^+_{DR}f(\bar{x}),$ may contain continuous linear functionals.

\vspace{2mm}

$C.$ Recall \cite{BonSha} that a function $f: X \to {\overline{\mathbb R}}$ is said to be \textit{directionally differentiable in the Hadamard sense}, or  \textit{Hadamard directionally differentiable} at a point $\bar{x} \in {\rm int}({\rm dom}f)$  if its radial directional derivative $f'(\bar{x}\,|\,d)$ exists for all $d \in X$  and, moreover,
$$
f'(\bar{x}\,|\,d) = \lim\limits_{(t,z) \to (0_+,d)}\frac{f(\bar{x} + tz) - f(\bar{x})}{t}.
$$
To emphasize that the directional derivative under consideration is the one in the Hadamard sense  we will use the notation $f'_H(\bar{x}\,|\,\cdot)$ instead of $f'(\bar{x}\,|\,\cdot).$

It is well-known (see, for instance, \cite{BonSha}) that $f'_H(\bar{x}\,|\,\cdot):X \to {\mathbb R}$ is a continuous function on $X.$

The sets $\partial^-_Hf(\bar{x}):= \{x^* \in X^* \mid x^*(d) \le f'_H(\bar{x}\,|\,d)\,\,\forall\,\,d \in X\}$ and $\partial^+_Hf(\bar{x}):= \{x^* \in X^* \mid x^*(d) \ge f'_H(\bar{x}\,|\,d)\,\,\forall\,\,d \in X\}$ are called, respectively, \textit{the Hadamard (directional) subdifferential} (also called \cite{Penot2,Ioffe17,Mord1} the Dini-Hadamard subdifferential) and \textit{the Hadamard (directional) superdifferential} of the function $f$ at $\bar{x}.$ Since each continuous linear minorant of $f'_H(\bar{x}\,|\,\cdot)$ is a maximal continuous sublinear minorant of it, we have $\partial^-_Hf(\bar{x}) \subseteq \partial^-_{DR}f(\bar{x}).$ Symmetrically, $\partial^+_Hf(\bar{x}) \subseteq \partial^+_{DR}f(\bar{x}).$
In fact, we have
$
\partial^-_Hf(\bar{x}) = \partial^-_{DR}f(\bar{x})\bigcap X^*\,\,\text{and}\,\,\partial^+_Hf(\bar{x}) = \partial^+_{DR}f(\bar{x})\bigcap X^*.
$

Consider the function $f(x_1,\,x_2) = |x_1| - |x_2|.$ It is directionally differentiable in the Hadamard sense at $(0,0)$ with $f'_H((0,0)\,|\,(d_1,\,d_2)) = |d_1| - |d_2|.$ It is easy to verify that both the $DR$-subdifferential and the $DR$-superdifferential of $f$ at $(0,\,0)$  are nonempty with ${\partial}^-_{DR}f((0,\,0)) = \{\alpha x_1 -|x_2|\,|\,-1 \le \alpha \le 1\}$ and ${\partial}^+_{DR}f((0,\,0)) = \{|x_1| - \alpha x_2\,|\,-1 \le \alpha \le 1\}$ At the same time $\partial^-_Hf((0,\,0)) = \varnothing$ and $\partial^+_Hf((0,\,0)) = \varnothing.$

\vspace{3pt}

The next theorem provides some calculus rules for $DR$-super(sub)differentials.

\begin{theorem}\label{th5.1}
Let $f,f_1,f_2$ be such functions for which there exist the $DR$-superdifferentials and the $DR$-subdifferentials at a point $\bar{x}.$  Then

$(i)$ $$\partial^+_{DR}(\lambda f)(\bar{x}) = \begin{cases} \lambda \partial^+_{DR}f(\bar{x}), & \text{when $\lambda > 0,$}\\ \lambda \partial^-_{DR}f(\bar{x}), & \text{when $\lambda < 0$};\end{cases}$$

$(ii)$ for any  $g_1 \in \partial^+_{DR}f_1(\bar{x})$ and any $g_2 \in \partial^+_{DR}f_2(\bar{x})$ there exists $g \in \partial^+_{DR}(f_1 + f_2)(\bar{x})$ such that $g(x) \le g_1(x) + g_2(x)$ for all $x \in X;$

$(iii)$ if $f_1(x) \le f_2(x)$ for all $x$ in a neighbourhood of $\bar{x}$ and $f_1(\bar{x}) = f_2(\bar{x})$, then for each $g_2 \in \partial^+_{DR}f_2(\bar{x})$ there exists $g_1 \in \partial^+_{DR}f_1(\bar{x})$ such that $g_1(x) \le g_2(x)$ for all $x \in X.$
\end{theorem}

The proof follows immediately from the definitions and the calculus of directional derivatives.

In the next two theorems we demonstrate some applications of $DR$-subdifferentials and $DR$-superdifferentials to extremal problems.

\begin{theorem}\label{th5.4}
Assume that the lower directional derivative $f^\downarrow (\bar{x}\,|\cdot )$ of a function $f:X \to {\mathbb R}$ at a point $\bar{x} \in X$ is bounded on the unit ball $B.$ If the point $\bar{x}$ is a local minimizer of $f$ over $X$, then $0_{X^*} \in \underline{\partial}^-_{DR}f(\bar{x})$ and, in addition,  $0_{X^*} \in \partial g$ for all $g \in \underline{\partial}^+_{DR}f(\bar{x})$.

Here $0_{X^*}$ is the null linear functional on $X,$ $\partial g$ is the subdifferential of a sublinear function $g$ in the sense of convex analysis.
\end{theorem}

\proof If  $\bar{x}$ is a local minimizer of $f$ over $X$, then $0 \le f^\downarrow (\bar{x}\,|\,d)$ for all $d \in X$ and, consequently, $0_{X^*} \in \underline{\partial}^-_{DR}f(\bar{x}).$ Besides, since for each $g \in  \underline{\partial}^+_{DR}f(\bar{x})$ the inequality $0 \le f^\downarrow (\bar{x}\,|\,d) \le g(d)\,\,\forall d \in X$ holds, we deduce that $0_{X^*} \in \partial g$ for all $g \in \underline{\partial}^+_{DR}f(\bar{x}).$

\hfill $\Box$

A function $f:X \to \overline{\mathbb R}$ is said \cite{BonSha} to be \textit{Fr\'{e}chet directionally differentiable at a point} $\bar{x} \in {\rm int}({\rm dom}f)$ if it is directionally differentiable at $\bar{x},$ its radial directional derivative $f'(\bar{x}\,|\,\cdot): X \to {\mathbb R}$ is continuous on $X,$ and the equality
\begin{equation}\label{e5.7}
\lim\limits_{z \to 0}\frac{f(\bar{x}+z) - f(\bar{x}) - f'(\bar{x}\,|\,z)}{\|z\|} =0
\end{equation}
holds.

\begin{theorem}\label{th5.2}
Assume that a function $f:X \to \overline{\mathbb R}$ is Fr\'{e}chet directionally differentiable at a point $\bar{x} \in {\rm int}({\rm dom}f).$  If there exists a real $\gamma > 0$ such that $\gamma B^* \subset \partial g$ for all $g \in {\partial}^+_{DR}f(\bar{x}),$ then the point $\bar{x}$ is a (srtict) local minimizer of $f$ over $X.$

Here $B^*$ is the unit ball in the space of continuous linear functionals defined  on $X,$ $\partial g$ is the subdifferential of a sublinear function $g$ in the sense of convex analysis.
\end{theorem}

\proof Since $f'(\bar{x}\,|\,\cdot)$ is continuous, then, by Theorem \ref{th4.7a}, we can present it in the form $f'(\bar{x}\,|\,d)= \inf\limits_{g \in {\partial}^+_{DR}f(\bar{x})}g(d)$ for all $d \in X.$  From the inclusion $\gamma B^* \subset \partial g$ holding for each $g \in {\partial}^+_{DR}f(\bar{x}),$ we get that $g(d) \ge \gamma \|d\|$ for all $d \in X$ and all $g \in {\partial}^+_{DR}f(\bar{x}).$ Consequently, $f'(\bar{x}\,|\,d) \ge \gamma \|d\|$ for all $d \in X.$ Due to the equality \eqref{e5.7} the last condition is sufficient for the point $\bar{x}$ to be a (strict) minimizer of $f$ over $X.$ \hfill $\Box$

\section{Conclusion}

The mapping $\Sigma^-_{Lip}: f \rightarrow \Sigma^-_{Lip}(f)$ introduced in the paper assigns to each real-valued function $f:X \to {\mathbb R}$ ($X$ is a  normed space) the uniquely determined (possibly empty) family $\Sigma^-_{Lip}(f)$ of maximal concave minorants of $f$ that are Lipschitz continuous on $X.$ By Theorem \ref{th4.4} the effective domain of the mapping $\Sigma^-_{Lip}$ is the collection of functions that are Lipschitz bounded from below on~$X.$ Moreover, if $\Sigma^-_{Lip}(f) \ne \varnothing$, then the upper envelope of $\Sigma^-_{Lip}(f)$ is just the lower semicontinuous clousure of $f.$ 

 As it was already observed in Section \ref{sec3} (see Example \ref{ex4.2}), the mapping $\Sigma^-_{Lip}: f \rightarrow \Sigma^-_{Lip}(f)$ can be considered as an extension of the Legendre-Fenchel transformation $f \to f^*$ from the class of lower semicontinuous convex functions to the collections of lower semicontinuous and Lipschitz bounded from below functions. In particular, the restriction of $\Sigma^-_{Lip}$ to the class of lower semicontinuous  p.h. functions is the extension of the classical Minkowski duality assigning to each continuous sublinear functions the Fenchel-Moreau subdifferential at the origin. These observations  gives hope that the results presented in this paper can serve as a basis for the development of a duality theory for functions that are much more complicated than convex ones.

 The families $\Sigma^-_{Lip}(f)$ and $\Sigma^+_{Lip}(f)$ are global characteristics of functions. The Demyanov-Rubinov subdifferential $\partial^-_{DR}f(x)$ and the Demyanov-Rubinov superdifferential $\partial^+_{DR}f(x)$ at a point $x$ introduced in the paper are their localizations based on directional derivatives of a function $f.$ At the same time, another approach to subdifferentiability, based on local approximations of  functions directly by subdifferential constructions rather than  on primary approximations by  directional derivatives, also are widely used in nonsmooth analysis. In particular, the Fr\'{e}chet subdifferential \cite{Kruger}, the limiting Kruger-Mordukhovich subdifferential \cite{Mord1}, and some others are introduced following this approach. The realization of the latter approach to localization of the characteristics $\Sigma^-_{Lip}(f)$ and $\Sigma^+_{Lip}(f)$ also is a perspective direction for further research.

\section*{Disclosure statement}

There are no conflicts of interest to disclose.

\section*{Funding}

The research was supported by the Belarussian State Research Programm under Grant `Conversion -- 1.04.01'.


\begin{thebibliography}{99}

\bibitem{Bourb}
Bourbaki~N. General Topology: Chapter 5-10. Berlin: Springer; 1989.

\bibitem{ET76}
Ekeland~I, Temam~R. Convex Analysis and Variational Problems,
Amsterdam: North-Holland, 1976. 

\bibitem{KutRub72}
Kutateladze~ SS, Rubinov~AM. Minkowski duality and its applications. Russian Math Surveys. 1972;27(3):137--192.

\bibitem{KutRub}
Kutateladze~ SS, Rubinov~AM. Minkowski duality and its applications. Novosibirsk: Nauka, 1976. Russian. -- 

\bibitem{GRC}
Glover~BM, Rubinov~AM, Craven~BD. Solvability theorems involving inf-convex functions. J Math Anal Appl. 1995;191:305--330.


\bibitem{Rub99}
Rubinov~AM. Supremal generators of spaces of homogeneous functions, In: Eberhard~A, Glover~B, Ralph~D, editors. Progress in Optimization: Contribution from Australasia, Dordrecht: Kluwer Academic Publishers. 1999. p. 91--100.

\bibitem{Rub00}
Rubinov~AM. Abstract convexity and global optimization. Dordrecht: Kluwer Academic Publishers, 2000.

\bibitem{PalRol}
Pallaschke~D, Rolewicz~S. Foundations of Mathematical Optimization (Convex analysis without linearity). Dordrecht: Kluwer Academic Publishers. 1997.

\bibitem{Singer}
Singer~I. Abstract Convex Analysis. New York (NY): Wiley-Interscience
Publication. 1997.

\bibitem{DR82}
Demyanov~VF, Rubinov~AM. [Elements of
quasidifferential calculus]. In: Demyanov~VF, Rubinov~AM, editors. Nonsmooth problems of optimization theory and control. Leningrad: Leningrad University Press, 1982. p. 5--127. Russian.

\bibitem{DR90}
Demyanov~VF, Rubinov~AM. Osnovy negladkogo analiza i kvazidifferentsial'noe ischislenie
[Foundations of nonsmooth analysis and quasidifferential calculus]. Moscow: Nauka Publ.; 1990. Russian.

\bibitem{DR95}
Demyanov~VF, Rubinov~AM.  Constructive Nonsmooth Analysis. Frankfurt: Verlag Peter Lang; 1995.

\bibitem{Ud2000}
Uderzo~A. Convex approximators, convexificators and exhausters: applications to constrained extremum problems. In: Demyanov~VF, Rubinov~AM, editors.
Quasidifferentiability and Related Topics. Dordrecht: Kluwer
Academic Publishers. 2000. p. 297--327.

\bibitem{Gor2017a}
Gorokhovik~VV. On the representation of upper semicontinuous functions defined on infinite-dimensional
normed spaces as lower envelopes of families of convex functions. Trudy Inst. Mat. Mekh. UrO RAN. 2017;23(1):88-102.


\bibitem{GS2011}
Gorokhovik~VV, Starovoitava~MA. Characteristic properties of primal exhausters for various classes of positively homogeneous
functions. Proceedings of the Institute of Mathematics. National Academy of Sciences of Belarus. 2011;19(2):12--25. Russian.

\bibitem{GorTr2016}
{Gorokhovik~VV, Trafimovich~MA}
Positively homogeneous functions revisited. J Optimiz Theory Appl. 2016;171(2):481-503.


\bibitem{Dem99}
Demyanov~VF. Exhausters of a positively homogeneous function.  Optimization. 1999;45:13--29.

\bibitem{DR2000}
Demyanov~VF.  Exhausters and convexificators --- new tools in nonsmooth analysis. In: Demyanov~VF, Rubinov~AM, editors.
Quasidifferentiability and Related Topics. Dordrecht: Kluwer
Academic Publishers. 2000. p. 85--137.

\bibitem{Rock}
Rockafellar~RT. Convex analysis. Princeton (NJ): Princeton University Press; 1970.

\bibitem{HULem1}
Hiriart-Urruty ~B, Lemarechal~C. Convex Analysis and Minimization Algorithms I. Fundamentals. Berlin: Springer; 1993. 

\bibitem{HULem2}	
Hiriart-Urruty~JB, Lemarechal~C. Convex Analysis and Minimization Algorithms II. Advanced Theory and Bundle Methods.  Berlin: Springer; 1993. 

\bibitem{Penot1}
Penot, J.-P. Calcul Sous-Differential et Optimization. J Functional Anal. 1978;2:248--276.

\bibitem{Penot2}
Penot~J.-P. Calculus without Derivatives. New York (NY): Springer; 2013. 

\bibitem {Ioffe17}
Ioffe~AD. Variational Analysis of Regular Mappings. Theory and Applications. Springer International Publishing AG, 2017. 

\bibitem{BonSha}
Bonnans~JF, Shapiro~A. Perturbation Analysis of Optimization Problems.
Berlin: Springer; 2000. 

\bibitem{Gor2017b}
Gorokhovik~VV. Demyanov-Rubinov subdifferentials of real-valued functions. In: Polyakova~LN, editor. 2017 Constructive Nonsmooth Analysis and Related Topics (Dedicated to the memory of V.F. Demyanov) (CNSA), Russia, Saint-Petersburg, May 22 – 27, 2017. Proceedings.  Institute of Electrical and Electronics Engineers (IEEE), 2017.  p. 122–125. DOI: 10.1109/cnsa.2017.7973962

\bibitem{Kelley}
Kelley~JL. General topology. New York (NY): Springer-Verlag, 1975.

\bibitem{Valent}
{Valentine~AF} Convex Sets. New York (NY): McGraw-Hill Book Company; 1964. 

\bibitem{Smith}
{Smith~CR.} A characterization of star-shaped sets. American Math.
Monthly. 1968;75:386.

\bibitem{GorZor}
{Gorokhovik~VV, Zorko~OI.} Piecewise affine functions and polyhedral sets. Optimization. 1994;31:209--221.

\bibitem{Cast1}
Castellani M. A dual representation for proper positively homogeneous functions. J Global Optim. 2000;16(4):393--400.


\bibitem{Cast2}
Castellani M. Dual representations of positively homogeneous functions. In: Demyanov~VF, Rubinov~AM, editors.
Quasidifferentiability and Related Topics. Dordrecht: Kluwer
Academic Publishers. 2000. p. 73--84.

\bibitem{Bauschke}
Bauschke~HH, Combetters~PL. Convex Analysis and Monotone Operator Theory in Hilbert Spaces. New York (NY): Springer; 2011.

\bibitem{RW98}
Rockafellar~RT, Wets~RJ-B. Variational analysis. Berlin: Springer; 1998.

\bibitem{HU}
Hiriart-Urruty~JB. Extension of Lipschitz functions. J Math Anal and Appl. 1980 77:539--554.

\bibitem{GT2015}
Gorokhovik~VV, Trafimovich~MA. On methods for converting exhausters of positively homogeneous functions. Optimization. 2016;65(3):589--608.

\bibitem{BorVan}
Borwein~JM, Vanderwerff~JD. Convex Functions: Constructions, Characterizations and Counterexamples. Cambridge: Cambridge University Press; 2010.

\bibitem{Mord1}
Mordukhovich~BS. Variational Analysis and Generalized Differentiation. I: Basic Theory. Berlin: Springer; 2005.

\bibitem{Kruger}
Kruger~AY. On Fr\'{e}chet subdifferentials. J Math. Sci. 2003;116(3):3325--3358.

\end{thebibliography}
\end{document}